\documentclass[11pt]{amsart}
\usepackage{fullpage}
\usepackage{amsmath, amssymb, enumerate, color, graphicx, bm, url}

\newtheorem{theorem}{Theorem}[section]

\newtheorem{corollary}[theorem]{Corollary}
\newtheorem{lemma}[theorem]{Lemma}

\newtheorem{fact}[theorem]{Fact}

\newtheorem{definition}[theorem]{Definition}

\theoremstyle{remark}
\newtheorem{remark}[theorem]{Remark}

\numberwithin{equation}{section}

\DeclareMathOperator*{\diam}{diam}
\DeclareMathOperator*{\E}{\mathbb{E}}

\DeclareMathOperator*{\sign}{sign}
\DeclareMathOperator*{\supp}{supp}

\def \R {\mathbb{R}}

\def \Z {\mathbb{Z}}
\def \P {\mathbb{P}}

\def \one {{\bf 1}}

\def \EE {\mathcal{E}}
\def \FF {\mathcal{F}}
\def \NN {\mathcal{N}}

\def \e {\varepsilon}

\def \l {\lambda}

\def \< {\langle}
\def \> {\rangle}
\def \^ {\widehat}

\newcommand{\norm}[1]{\left \|#1\right \|}

\newcommand{\twonorm}[1]{\norm{#1}_2}

\newcommand{\fronorm}[1]{\norm{#1}_F}
\newcommand{\abs}[1]{\left | #1 \right |}
\renewcommand{\Pr}[1]{\P \left\{ #1 \rule{0mm}{3mm}\right\}}

\usepackage{mathbbol}

\title{Dimension reduction by random hyperplane tessellations}

\author{Yaniv Plan}
\author{Roman Vershynin}

\date{\today}

\address{Department of Mathematics,
  University of Michigan,
  530 Church St.,
  Ann Arbor, MI 48109, U.S.A.}

\email{\{yplan,romanv\}@umich.edu}


\thanks{Y.P. is supported by an NSF Postdoctoral Research Fellowship under award No. 1103909. 
  R.V. is supported by NSF grants DMS 0918623 and 1001829.}

\begin{document}

\begin{abstract}

  Given a subset $K$ of the unit Euclidean sphere, we estimate the minimal number $m = m(K)$ of hyperplanes
  that generate a uniform tessellation of $K$, in the sense that  
  the fraction of the hyperplanes separating any pair $x,y \in K$ 
  is nearly proportional to the Euclidean distance between $x$ and $y$.
  Random hyperplanes prove to be almost ideal for this problem; they achieve the almost
  optimal bound $m = O(w(K)^2)$ where $w(K)$ is the Gaussian mean width of $K$. 
  Using the map that sends $x \in K$ to the sign vector with respect to the hyperplanes, 
  we conclude that every bounded subset $K$ of $\R^n$ embeds into the 
  Hamming cube $\{-1,1\}^m$ with a small distortion 
  in the Gromov-Haussdorff metric. Since for many sets $K$ one has $m = m(K) \ll n$, 
  this yields a new discrete mechanism of dimension reduction for sets in Euclidean spaces.
\end{abstract}

\maketitle

\section{Introduction}

Consider a bounded subset $K$ of $\R^n$. We would like to find an arrangement of $m$ affine hyperplanes in $\R^n$ 
that cut through $K$ as evenly as possible; see Figure~\ref{fig:tessellation} for an illustration. 
The intuitive notion of an ``even cut'' can be expressed more formally in the following way:
The fraction of the hyperplanes separating any pair $x,y \in K$ 
should be proportional (up to a small additive error) to the Euclidean distance between $x$ and $y$. 
What is the smallest possible number $m=m(K)$ of hyperplanes with this property? 
Besides having a natural theoretical appeal, this question is directly motivated by a certain problem of information theory 
which we will describe later.

\begin{figure}[htp]		
  \centering \includegraphics[height=2.7cm]{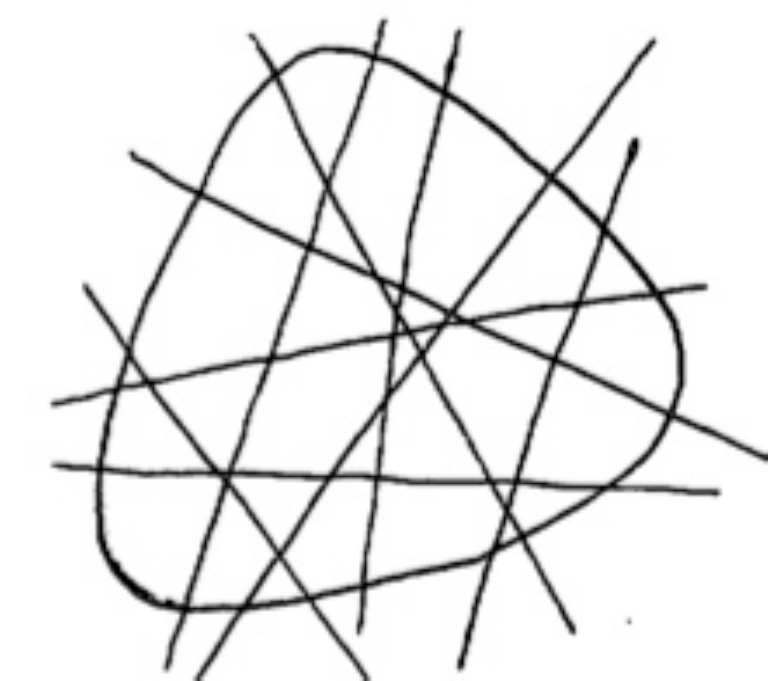} 
  \caption{A hyperplane tessellation of a set in the plane}
  \label{fig:tessellation}
\end{figure}

In the beginning it will be most convenient to work with subsets $K$ of the unit Euclidean sphere $S^{n-1}$, but we will lift this restriction later. 
Let $d(x,y)$ denote the normalized geodesic distance on $S^{n-1}$, so the distance between the opposite points
on the sphere equals $1$. 
A (linear) hyperplane in $\R^n$ can be expressed as $a^\perp$ for some $a \in \R^n$. We say that points $x,y \in \R^n$
are separated by the hyperplane\footnote{For convenience of presentation we prefer the sign function to take values
$\{-1,1\}$, so we define it as $\sign(t)=1$ for $t \ge 0$ and $\sign(t) = -1$ for $t<0$.}
if $\sign\< a,x\> \ne \sign\< a,y\> $.

\begin{definition}[Uniform tessellation]	
  Consider a subset $K \subseteq S^{n-1}$ 
  and an arrangement of $m$ hyperplanes in $\R^n$.
  Let $d_A(x,y)$ denote the fraction of the hyperplanes that separate points $x$ and $y$ in $\R^n$.
  Given $\delta>0$, we say that the hyperplanes provide a {\em $\delta$-uniform tessellation} of $K$ 
  if 
  \begin{equation}							\label{eq:uniform tessellation}
  | d_A(x,y) - d(x,y) | \le \delta, \quad x,y \in K.
  \end{equation}
\end{definition}

The main result of this paper is a bound on the minimal number $m = m(K,\delta)$ of hyperplanes that provide a uniform tessellation of a set $K$.
It turns out that for a fixed accuracy $\delta$, 
an almost optimal estimate on $m$ depends only on one global parameter of $K$, namely the mean width.
Recall that the {\em Gaussian mean width} of $K$ is defined as 
\begin{equation}							\label{eq:mean width}
w(K) = \E \sup_{x \in K} |\< g,x \> |
\end{equation}
where $g \sim \NN(0, I_n)$ is a standard Gaussian random vector in $\R^n$. 

\begin{theorem}[Random uniform tessellations] 				\label{thm:tessellations}
  Consider a subset $K \subseteq S^{n-1}$ and let $\delta > 0$. 
  Let 
  $$
  m \ge C \delta^{-6} w(K)^2
  $$ 
  and consider an arrangement of $m$ independent random hyperplanes in $\R^n$ 
  uniformly distributed according to the Haar measure. 
  Then with probability at least $1-2\exp(-c \delta^2 m)$, these hyperplanes provide a $\delta$-uniform tessellation of $K$. 
  Here and later $C, c$ denote positive absolute constants. 
\end{theorem}

\begin{remark}[Tessellations in stochastic geometry]
  By the rotation invariance of the Haar measure, it easily follows that 
  $\E d_A(x,y) = d(x,y)$ for each pair $x,y \in \R^n$.
  Theorem~\ref{thm:tessellations} states that with high probability, $d_A(x,y)$ almost matches
  its expected value {\em uniformly} over all $x,y \in K$.  
  This observation highlights the principal difference between the problems studied in this paper 
  and the classical problems on random hyperplane tessellations studied in stochastic geometry. 
  The classical problems concern the shape of a {\em specific} cell (usually the one containing the origin)
  or certain {\em statistics} of cells (e.g. ``how many cells have volume greater than a fixed number''?), see \cite{Calka}.
  In contrast to this, the concept of uniform tessellation we propose his paper concerns {\em all} cells simultaneously; 
  see Section~\ref{s:cells} for a vivid illustration. 
\end{remark}

\subsection{Embeddings into the Hamming cube}

Theorem~\ref{thm:tessellations} has an equivalent formulation in the context of metric embeddings. It yields that 
{\em every subset $K \subseteq S^{n-1}$ can be almost isometrically embedded into the Hamming cube $\{-1,1\}^m$ with 
$m = O(w(K)^2)$}. 

\smallskip

To explain this statement, let us recall a few standard notions. 
An $\e$-isometry (or almost isometry) between metric spaces $(X,d_X)$ and $(Y, d_Y)$
is a map $f : X \to Y$ which satisfies 
$$
| d_Y(f(x),f(x')) - d_X(x,x') | \le \e, \quad x,x' \in X,
$$
and such that for every $y \in Y$ one can find $x \in X$ satisfying $d_Y(y,f(x)) \le \e$.
A map $f : X \to Y$ is an $\e$-isometric embedding of $X$ into $Y$ if the map $f : X \to f(X)$ is an $\e$-isometry
between $(X,d_X)$ and the subspace $(f(X),d_Y)$. 
It is not hard to show that $X$ can be $2 \e$-isometrically embedded into $Y$ (by means of a suitable map $f$) if 
$X$ has the Gromov-Haussdorff distance at most $\e$ from some subset of $Y$.  Conversely, if there is an $\e$-isometry between $X$ and $f(X)$ then the Gromov-Haussdorff distance between $X$ and $f(X)$ is bounded by $\e$.

Finally, recall that the Hamming cube is the set $\{-1,1\}^m$ with the (normalized) Hamming distance
$d_{H}(u,v) = \frac{1}{m} \sum_{i=1}^m \mathbb{1}_{\{u_i \neq v_i\}} = $ the fraction of the coordinates where $u$ and $v$ are different. 

\smallskip

An arrangement of $m$ hyperplanes in $\R^n$ defines a {\em sign map} 
$f : \R^n \to \{-1,1\}^m$ which sends $x \in \R^n$ to the sign vector of the orientations 
of $x$ with respect to the hyperplanes. The sign map is uniquely defined up to the isometries of the Hamming cube.
Let $a_1,\ldots, a_m \in \R^n$ be normals of the hyperplanes, 
and consider the $m \times n$ matrix $A$ with rows $a_i$.
The sign map can be expressed as 
$$
f(x) = \sign Ax, \quad f: \R^n \to \{-1,1\}^m,
$$
where $\sign Ax$ denotes the vector of signs of the coordinates $\< a_i, x\> $ of $Ax$.
The fraction $d_A(x,y)$ of the hyperplanes that separate points $x$ and $y$ thus equals
$$
d_A(x,y) = d_H(\sign Ax, \sign Ay), \quad x,y \in \R^n.
$$
Then looking back at the definition of uniform tessellations, we observe the following fact:

\begin{fact}[Embeddings by uniform tessellations]				\label{fact:embeddings}
  Consider a $\delta$-uniform tessellation of a set $K \subseteq S^{n-1}$ by $m$ hyperplanes. 
  Then the set $K$ (with the induced geodesic distance) 
  can be $\delta$-isometrically embedded into the Hamming cube $\{-1,1\}^m$. 
  The sign map provides such an embedding. \qed
\end{fact}

This allows us to state Theorem~\ref{thm:tessellations} as follows:

\begin{theorem}[Embeddings into the Hamming cube] 				\label{thm:embeddings}
  Consider a subset $K \subseteq S^{n-1}$ and let $\delta > 0$. 
  Let 
  $$
  m \ge C \delta^{-6} w(K)^2.
  $$ 
  Then $K$ can be $\delta$-isometrically embedded into the Hamming cube $\{-1,1\}^m$.
  
  Moreover, let $A$ be an $m \times n$ random matrix with independent $\NN(0,1)$ entries. 
  Then with probability at least $1-2\exp(-c \delta^2 m)$, the sign map 
  \begin{equation}							\label{eq:f}
  f(x) = \sign Ax, \quad f : K \to \{-1,1\}^m
  \end{equation}
  is an $\delta$-isometric embedding. \qed
\end{theorem}

\subsection{Almost isometry of $K$ and the tessellation graph.}

The image of the sign map $f$ in \eqref{eq:f} has a special meaning.
When the Hamming cube $\{-1,1\}^m$ is viewed as a graph (in which two points $u$, $v$ 
are connected if they differ in exactly one coordinate), the image of $f$ 
defines a subgraph of $\{-1,1\}^m$, which is called the {\em tessellation graph} of $K$. 
The tessellation graph has a vertex for each cell and an edge for each pair of adjacent cells, 
see Figure~\ref{fig:tessellation-graph}. Notice that the graph distance in the tessellation graph 
equals the number of hyperplanes that separate the two cells. Therefore the
definition of a uniform tessellation yields:

\begin{fact}[Graphs of uniform tessellations]					\label{fact:graph}
  Consider a $\delta$-uniform tessellation of a set $K \subseteq S^{n-1}$. 
  Then $K$ is $\delta$-isometric to the tessellation graph of $K$. \qed
\end{fact}

Hence we can read the conclusion of Theorem~\ref{thm:tessellations} as follows:
{\em $K$ is $\delta$-isometric to the graph of its tessellation by $m$ random hyperplanes, where $m \sim \delta^{-6} w(K)^2$}.
\begin{figure}[htp]		
  \centering \includegraphics[height=2.7cm]{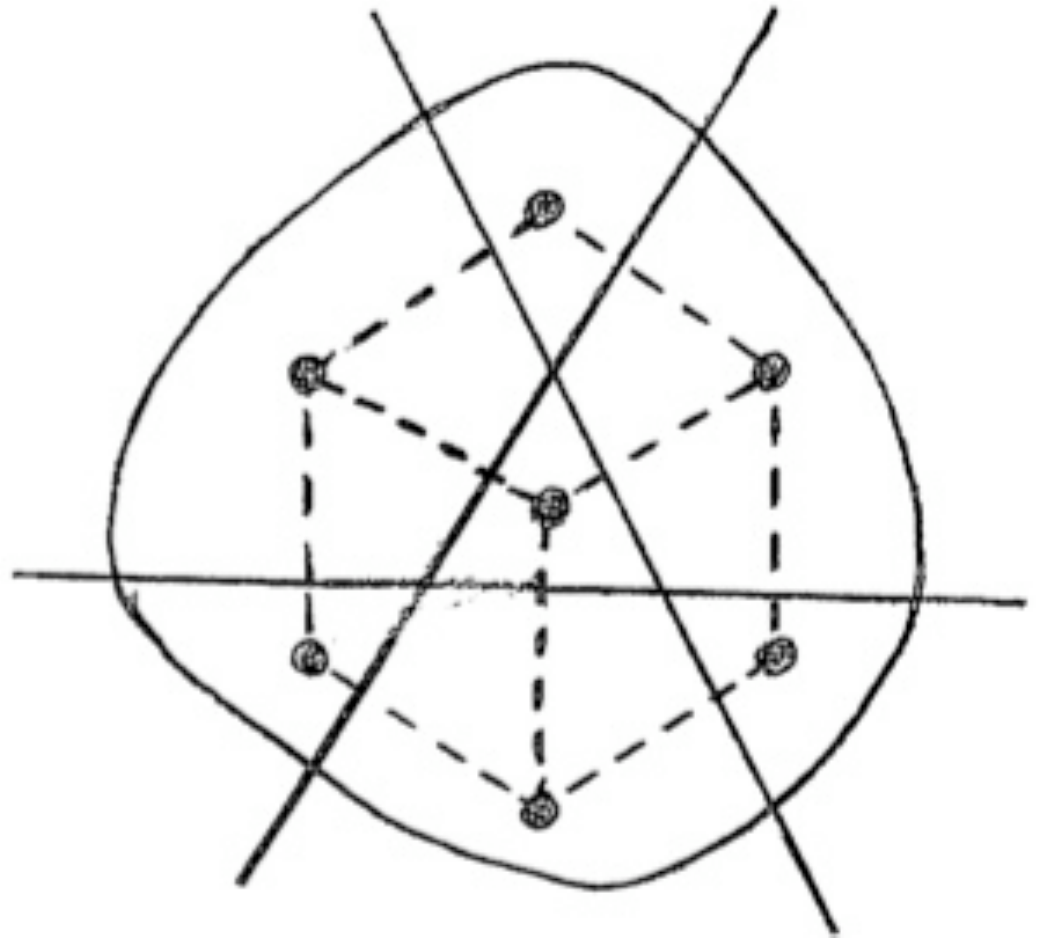} 
  \caption{The graph of a tessellation of a set in the plane. The dashed lines represent the edges.}
  \label{fig:tessellation-graph}
\end{figure}

\subsection{Computing mean width}

Powerful methods to estimate the mean width $w(K)$ have been developed in connection 
with stochastic processses. These methods include Sudakov's and Dudley's inequalities 
which relate $w(K)$ to the covering numbers of $K$ in 
the Euclidean metric, and the sharp technique of majorizing measures (see \cite{LT, T Chaining}).

Mean width has a simple (and known) geometric interpretation.
By the rotational invariance of the Gaussian random vector $g$ in \eqref{eq:mean width}, 
one can replace $g$ with a random vector $\theta$ that is uniformly distributed on $S^{n-1}$, as follows:
$$
w(K) = c_n \sqrt{n} \cdot \bar{w}(K), \quad \text{where} \quad \bar{w}(K) = \E \sup_{x \in K} |\< \theta,x\> |.
$$
Here $c_n$ are numbers that depend only on $n$ and such that $c_n \le 1$ and  
$\lim_{n \to \infty} c_n = 1$.  
We may refer to $\bar{w}(K)$ as the {\em spherical mean width} of $K$. Let us assume for simplicity that $K$ is symmetric with respect to the origin. Then $2\sup_{x \in K} |\< \theta,x\> |$ is the width of $K$ in the direction $\theta$, which is the distance 
between the two supporting hyperplanes of $K$ whose normals are $\theta$. 
The spherical mean width $\bar{w}(K)$ is then twice the average width of $K$ over all directions.

\subsection{Dimension reduction}

Our results are already non-trivial in the particular case $K = S^{n-1}$. 
Since $w(S^{n-1}) \le \sqrt{n}$, Theorems~\ref{thm:tessellations} and \ref{thm:embeddings} hold with 
$m \sim n$.  
But more importantly, many interesting sets $K \subset S^{n-1}$ satisfy $w(K) \ll \sqrt{n}$ 
and therefore make our results hold with $m \sim w(K)^2 \ll n$. 
In such cases, one can view the sign map $f(x) = \sign Ax$ in Theorem~\ref{thm:embeddings} as a dimension 
reduction mechanism that transforms an $n$-dimensional set $K$ into a subset of $\{-1,1\}^m$.

A heuristic reason why dimension reduction is possible is that the quantity $w(K)^2$ measures 
the {\em effective dimension} of a set $K \subseteq S^{n-1}$. 
  The
 effective dimension $w(K)^2$ of a set $K \subseteq S^{n-1}$
is always bounded by the algebraic dimension, but it may be much smaller
and it is robust with respect to perturbations of $K$. 
In this regard, the notion of effective dimension is parallel to the notion of effective rank of a matrix
from numerical linear algebra (see e.g. \cite{RV JACM}).
With these observations in mind, it is not surprising that the ``true'', effective dimension of $K$ 
would be revealed (and would be the only obstruction according to Theorem~\ref{thm:embeddings})
when $K$ is being squeezed into a space of smaller dimension.

\smallskip

Let us illustrate dimension reduction on the example of finite sets $K \subset S^{n-1}$.
Since $w(K) \le C \sqrt{\log |K|}$ (see e.g. \cite[(3.13)]{LT}), Theorem~\ref{thm:embeddings} holds with 
$m \sim \log |K|$, and we can state it as follows.

\begin{corollary}[Dimension reduction for finite sets]					\label{cor:JL}
  Let $K \subset S^{n-1}$ be a finite set. Let $\delta>0$ and $m \ge C \delta^{-6} \log |K|$. 
  Then $K$ can be $\delta$-isometrically embedded into the Hamming cube $\{-1,1\}^m$. \qed
\end{corollary}

This fact should be compared to the {\em Johnson-Lindenstrauss lemma} for finite subsets $K \subset \R^n$ 
(\cite{JL}, see \cite[Section~15.2]{Matousek}) which states that if $m \ge C \delta^{-2} \log |K|$ then 
$K$ can be Lipschitz embedded into $\R^m$ as follows:
$$
\big| \|\bar{A}x - \bar{A}x'\|_2 - \|x-x'\|_2 \big| \le \delta \|x-x'\|_2, \quad x,x' \in K.
$$
Here $\bar{A} = m^{-1/2} A$ is the rescaled random Gaussian matrix $A$ from Theorem~\ref{thm:embeddings}.  Note that while the Johnson-Lindenstrauss lemma involves a Lipschitz embedding from $\R^n$ to $\R^m$, it is generally impossible to provide a Lipschitz embedding from subsets of $\R^n$ to the Hamming cube
(if there are points $x,x' \in K$ that are very close to each other); this is why we consider $\delta$-isometric embeddings.

Like  the Johnson-Lindenstrauss lemma, Corollary~\ref{cor:JL} can be proved directly by combining concentration inequalities
for $d_A(x,y)$ with a union bound over $|K|^2$ pairs $(x,y) \in K \times K$.  In fact, this method of proof allows for the weaker requirement $m \geq C \delta^{-2} \log |K|$. However, as we discuss later, 
this argument cannot be generalized in a straightforward way to prove Theorem~\ref{thm:embeddings} for general sets $K$.
The Hamming distance $d_A(x,y)$ is highly discontinuous, which makes it difficult to extend estimates from 
points $x,y$ in an $\e$-net of $K$ to nearby points.

\subsection{Cells of uniform tessellations}					\label{s:cells}

We mentioned two nice features of uniform tessellations in Facts~\ref{fact:embeddings} and \ref{fact:graph}.
Let us observe one more property: all cells of a uniform tessellation have small diameter. 
Indeed, $d_A(x,y) = 0$ iff points $x,y$ are in the same cell, so by \eqref{eq:uniform tessellation} we have:

\begin{fact}[Cells are small]				\label{fact:cells}
  Every cell of a $\delta$-uniform tessellation has diameter at most $\delta$. \qed
\end{fact}

With this, Theorem~\ref{thm:tessellations} immediately implies the following: 

\begin{corollary}[Cells of random uniform tessellations] 				\label{cor:cells}
  Consider a tessellation of a subset $K \subseteq S^{n-1}$ by $m \geq C\delta^{-6} w(K)^2$ random hyperplanes.
  Then, with probability at least $1-\exp(-c \delta^2 m)$, all cells of the tessellation have diameter at most $\delta$.
\end{corollary}

This result has also a direct proof, which moreover gives 
a slightly better bound $m \sim \delta^{-4} w(K)^2$. We present this 
``curvature argument'' in Section~\ref{sec:curvature}.

\subsection{Uniform tessellations in $\R^n$}

So far, we only worked with subsets $K \subseteq S^{n-1}$. 
It is not difficult to extend our results to bounded sets $K \subset \R^n$. 
This can be done by embedding such a set $K$ into $S^{n}$ (the sphere in one more dimension) 
with small bi-Lipschitz distortion. 
This elementary argument is presented in Section~\ref{sec:Rn}, and it yields
the following version of Theorem~\ref{thm:tessellations}:

\begin{theorem}[Random uniform tessellations in $\R^n$] 				\label{thm:tessellations Rn}
  Consider a bounded subset $K \subset \R^n$ with $\diam(K) = 1$.
  Let 
  \begin{equation}							\label{eq:tessellations Rn m}
  m \ge C \delta^{-12} w(K-K)^2.
  \end{equation}
  Then there exists an arrangement of $m$ affine hyperplanes in $\R^n$ and a scaling factor $\l>0$ 
  such that 
  $$
  \big| \l \cdot d_A(x,y) - \|x-y\|_2 \big| \le \delta, \quad x,y \in K.
  $$
  Here $d_A(x,y)$ denotes the fraction of the affine hyperplanes that separate $x$ and $y$.
\end{theorem}

\begin{remark}[Mean width in $\R^n$]				\label{rem:K-K}
  While the quantity $w(K-K)$ appearing in \eqref{eq:tessellations Rn m} 
  is clearly bounded by $2w(K)$, it is worth noting that the quantity $w(K-K)$ 
  captures more accurately than $w(K)$ 
  the geometric nature of the ``mean width'' of $K$. Indeed, 
  $w(K-K) = \E h(g)$ where $h(g) = \sup_{x \in K} \< g,x\> - \inf_{x \in K} \< g,x\> $ is the distance 
  between the two parallel supporting hyperplanes of $K$ orthogonal to the random direction $g$, scaled by $\|g\|_2$.
\end{remark}

\subsection{Optimality}

The main object of our study is $m(K) = m(K,\delta)$, the smallest number of hyperplanes that provide a $\delta$-uniform tessellation of 
a set $K \subseteq S^{n-1}$. One has
\begin{equation}							\label{eq:upper lower}
\log_2 N(K,\delta) \le m(K,\delta) \le C \delta^{-6} w(K)^2,
\end{equation}
where $N(K,\delta)$ denotes the covering number of $K$, i.e.~the smallest number of balls of radius $\delta$ that cover $K$.
The upper bound in \eqref{eq:upper lower} is the conclusion of Theorem~\ref{thm:tessellations}. The lower bound holds because
a $\delta$-uniform tessellation provides a decomposition of $K$ into at most $2^m$ cells
each of which lies in a ball of radius $\delta$ by Fact~\ref{fact:cells}.

To compare the upper and lower bounds in \eqref{eq:upper lower}, recall Sudakov's inequality \cite[Theorem~3.18]{LT}
that yields
$$
\log N(K,\delta) \le C \delta^{-2} w(K)^2.
$$
While Sudakov's inequality cannot be reversed in general, there are many situations where it is sharp. 
Moreover, according to Dudley's inequality (see \cite[Theorem 11.17]{LT} and \cite[Lemma~2.33]{M}), 
Sudakov's inequality can always be reversed for some scale $\delta>0$ and up to a logarithmic factor in $n$.
(See also \cite{LMPT} for a discussion of sharpness of Sudakov's inequality.)
So the two sides of \eqref{eq:upper lower} are often close to each other, but there is in general some gap. 
We conjecture that the optimal estimate is
$$
c w(K)^2 \le \sup_{\delta>0} \delta^2 m(K,\delta) \le C  w(K)^2,
$$
so the mean width of $K$ seems to be completely responsible for the uniform tessellations of $K$. 

\smallskip

Note that the lower bound in \eqref{eq:upper lower} holds in greater generality. 
Namely, it is not possible to have $m < \log_2 N(K,\delta)$ for {\em any} decomposition 
of $K$ into $2^m$ pieces of diameter at most $\delta$. However, from the upper bound 
we see that with a slightly larger value $m \sim w(K)^2$,
{\em an almost best decomposition of $K$ is achieved by a random hyperplane tessellation}.

\smallskip

In this paper we have not tried to optimize the dependence of $m(K,\delta)$ on $\delta$. This interesting
problem is related to the open question on the optimal dependence on distortion in Dvoretzky's theorem. 
We comment on this in Section~\ref{sec:Dvoretzky}.

\subsection{Related work: embeddings of $K$ into normed spaces}

Embeddings of subsets $K \subseteq S^{n-1}$ into normed spaces were studied in geometric functional 
analysis \cite{KM, Sch}.
In particular, Klartag and Mendelson \cite{KM} were concerned with embeddings into $\ell_2^m$. 
They showed that for $m \ge C \delta^{-2} w(K)^2$ there exists a linear map $A: \R^n \to \R^m$ 
such that 
$$
\big| m^{-1/2} \|Ax\|_2 - 1 \big| \le \delta, \quad x \in K.
$$ 
One can choose $A$ to be an $m \times n$ random matrix with Gaussian entries as in Theorem~\ref{thm:embeddings},
or with sub-gaussian entries.   
Schechtman \cite{Sch} gave a simpler argument for a Gaussian matrix, which also works for embeddings into 
general normed spaces $X$. In the specific case of $X = \ell_1^m$, Schechtman's result states that for $m \ge C \delta^{-2} w(K)^2$
one has
$$
\big| m^{-1} \|Ax\|_1 - 1 \big| \le \delta, \quad x \in K.
$$ 
This result also follows from Lemma~\ref{lem:concentration} below. 

\subsection{Related work: one-bit compressed sensing}

Our present work was motivated by the development of {\em one-bit compressed sensing} in \cite{BB, JLBB, PV}
where Theorem~\ref{thm:embeddings} is used in the following context.
The vector $x$ represents a signal; the matrix $A$ represents a measurement map $\R^n \to \R^m$
that produces $m \ll n$ linear measurements of $x$; taking the sign of $Ax$ represents
quantization of the measurements (an extremely coarse, one-bit quantization).
The problem of one-bit compressed sensing is to recover the signal $x$ 
from the quantized measurements $f(x) = \sign Ax$. 

The problem of one-bit compressed sensing was introduced by Boufounos and Baraniuk \cite{BB}.
Jacques, Laska, Boufounos and Baraniuk \cite{JLBB} realized a connection of this problem to uniform tessellations 
of the set of sparse signals $K = \{x \in S^{n-1}:\; |\supp(x)| \le s\}$, 
and to almost isometric embedding of $K$ into the Hamming cube $\{-1,1\}^m$. 
For this set $K$, they proved Corollary~\ref{cor:cells} with $m \sim \delta^{-1} s \log(n/\delta)$ and 
a version of Theorem~\ref{thm:embeddings} for $m \sim \delta^{-2} s \log(n/\delta)$. The authors of the present paper
analyzed in \cite{PV} a bigger set of ``compressible'' signals $K' = \{ x \in S^{n-1}:\; \|x\|_1 \le \sqrt{s} \}$
and proved for $K'$ a version of Corollary~\ref{cor:cells} with $m \sim \delta^{-4} s \log(n/s)$. 
Since the mean widths of both sets $K$ and $K'$ are of the order $\sqrt{s \log(n/s)}$, 
Theorem~\ref{thm:embeddings} holds for these sets with $m \sim \delta^{-6} s \log(n/s)$. 
In other words, apart from the dependence of $\delta$ (which is an interesting problem), 
the prior results follow as partial cases from Theorem~\ref{thm:embeddings}.

It is important to note that Theorem~\ref{thm:embeddings} addresses only the theoretical aspect 
of one-bit compressed sensing problem, which guarantees that the quantized measurement map $f(x) = \sign Ax$ 
well preserves the geometry of signals. 
But one also faces an algorithmic challenge -- how to efficiently recover $x$ from $f(x)$, and specifically in polynomial time. 
We will not touch on this algorithmic aspect here but rather refer the reader to \cite{PV} 
and to our forthcoming work which is based on the results of this paper.

\subsection{Related work: locality-sensitive hashing}
\textit{Locality-sensitive hashing} is a method of dimension reduction.  One takes a set of high-dimensional vectors in $\R^n$ and the goal is to hash nearby vectors to the same bin with high probability.  More generally, one may desire that the distance between bins be nearly proportional to the distance between the original items.  There have been a number of papers which suggest to create such mappings onto the Hamming cube \cite{GW, AP, Charikar, ASS, KOR}, some of which use a random hyperplane tessellation as defined in this paper.  The new challenge considered herein is to create a locality-sensitve hashing for an infinite set.

\subsection{Overview of the argument}

Let us briefly describe our proof of the results stated above. 
Since the distance in the Hamming cube $\{-1,1\}^m$ can be expressed as $(2m)^{-1}\|x-y\|_1$, the 
Hamming cube is isometrically embedded in $\ell_1^m$. 
Before trying to embed $K \subseteq S^{n-1}$ into the Hamming cube as claimed in Theorem~\ref{thm:embeddings},
we shall make a simpler step and embed $K$ almost isometrically into the bigger space $\ell_1^m$ with $m \sim \delta^{-2} w(K)^2$. 
A result of this type was given by Schechtman \cite{Sch}. 
In Section~\ref{sec:into ell1} we prove a similar result by a simple and direct argument in probability in Banach spaces.

Our next and non-trivial step is to re-embed the set from $\ell_1^m$ into its subset, the Hamming cube $\{-1,1\}^m$.
In Section~\ref{sec:curvature} we give a simple ``curvature argument'' that allows us to deduce
Corollary~\ref{cor:cells} on the diameter of cells, and even with a better dependence on $\delta$, namely $m \sim \delta^{-4} w(K)^2$.
However, a genuine limitation of the curvature argument makes it too weak to deduce Theorem~\ref{thm:tessellations} this way.

We instead attempt to prove Theorem~\ref{thm:tessellations} by an $\e$-net argument, which typically proceeds as follows:
(a) show that $d_A(x,y) \approx d(x,y)$ holds for a fixed pair $x,y \in K$ with high probability;
(b) take the union bound over all pairs $x,y$ in an finite $\e$-net $N_\e$ of $K$;  
(c) extend the estimate from $N_\e$ to $K$ by approximation.
Unfortunately, as we indicate in Section~\ref{sec:soft} the approximation step (c) must fail due to the discontinuity 
of the Hamming distance $d_A(x,y)$. 

A solution proposed in \cite{B, JLBB} was to choose $\e$ so small that none of the random hyperplanes 
pass near points $x,y \in N_\e$ with high probability.  
This strategy was effective for the set $K = \{x \in S^{n-1}:\; |\supp(x)| \le s\}$
because the covering number of this specific set $K$ has a mild (logarithmic) dependence on $\e$, namely 
$\log N(K,\e) \le s \log (Cn/\e s)$. 
However, adapting this strategy to general sets $K$ would cause our estimate on $m$ to increase by a factor of $n$.

The solution we propose in the present paper is to ``soften'' the Hamming distance; see Section~\ref{sec:soft} for the precise notion.
The {\em soft Hamming distance} enjoys some continuity properties as described in Lemmas~\ref{lem:continuity} and \ref{lem:continuity L1}.
In Section~\ref{sec:proof tessellations} we develop the $\e$-net argument for the soft Hamming distance.
Interestingly, the approximation step (c) for the soft Hamming distance will be based on
the embedding of $K$ into $\ell_1^m$, which incidentally was our point of departure.

\subsection{Notation}

Throughout the paper, $C$, $c$, $C_1$, etc.~denote positive absolute constants whose values may change from line to line. 
For integer $n$, we denote $[n] = \{1,\ldots,n\}$.
The $\ell_p$ norms of a vector $x \in \R^n$ for $p \in \{0,1,2,\infty\}$ are defined 
as\footnote{Note that, strictly speaking, $\|\cdot\|_0$ is not a norm on $\R^n$.}
$$
\|x\|_0 = |\supp(x)| = | \{ i \in [n]: x(i) \ne 0 \} |, \;
\|x\|_1 = \sum_{i=1}^n |x_i|, \;
\|x\|_2 = \big( \sum_{i=1}^n x_i^2 \big)^{1/2}, \;
\|x\|_\infty = \max_{i \in [n]} |x_i|.
$$
We shall work with normed spaces $\ell_p^n = (\R^n, \|\cdot\|_p)$ for $p \in \{1,2,\infty\}$.
The unit Euclidean ball in $\R^n$ is denoted $B_2^n = \{ x \in \R^n:\; \|x\|_2 \le 1 \}$ and 
the unit Euclidean sphere is denoted $S^{n-1} = \{ x \in \R^n:\; \|x\|_2 = 1 \}$.

As usual, $\NN(0,1)$ stands for the univariate normal distribution with zero mean and unit variance,
and $\NN(0,I_n)$ stands for the multivariate normal distribution in $\R^n$ with zero mean and 
whose covariance matrix is identity $I_n$.

\section{Embedding into $\ell_1$}					\label{sec:into ell1}

\begin{lemma}[Concentration]				\label{lem:concentration}
  Consider a bounded subset $K \subset \R^n$ and independent random vectors 
  $a_1,\ldots,a_m \sim \NN(0,I_n)$ in $\R^n$.
  Let 
  $$
  Z = \sup_{x \in K} \Big| \frac{1}{m} \sum_{i=1}^m |\< a_i,x\> | - \sqrt{\frac{2}{\pi}} \|x\|_2 \Big|.
  $$
  
  (a) One has 
  \begin{equation}							\label{eq:L1 embedding}
  \E Z
  \le \frac{4 w(K)}{\sqrt{m}}.
  \end{equation}
  
  (b) The following deviation inequality holds: 
  \begin{equation}
  \label{eq:isoperimetric deviation inequality}
  \Pr{Z > \frac{4 w(K)}{\sqrt{m}} + u}
  \le 2 \exp \Big(-\frac{m u^2}{2 d(K)^2} \Big), \qquad u>0
  \end{equation}
  where $d(K) = \max_{x \in K} \|x\|_2$. 
\end{lemma}

\begin{proof}
(a) Note that $\E |\< a_i,x\> | = \sqrt{\frac{2}{\pi}} \|x\|_2$ for all $i$. 
Let $\e_1, \hdots, \e_m$ be a sequence of iid rademacher random variables.
A standard symmetrization argument (see \cite[Lemma 6.3]{LT}) followed by the contraction principle (see \cite[Theorem 4.12]{LT})
yields that 
$$
\E Z \le 2\E \sup_{x \in K} \Big| \frac{1}{m} \sum_{i=1}^m \e_i \abs{\< a_i,x\>} \Big| 
\le 4\E \sup_{x \in K} \Big| \frac{1}{m} \sum_{i=1}^m \e_i \< a_i,x\> \Big| 
= 4\E \sup_{x \in K} \Big| \Big\langle \frac{1}{m} \sum_{i=1}^m \e_i a_i, x \Big\rangle \Big|.
$$
By the rotational invariance of the Gaussian distribution, 
$\frac{1}{m} \sum_{i=1}^m \e_i a_i$ is distributed identically with $g/\sqrt{m}$ where $g \sim \NN(0,I_n)$. 
Therefore 
$$
\E Z \le \frac{4}{\sqrt{m}} \E \sup_{x \in K} |\< g,x\> |
= \frac{4 w(K)}{\sqrt{m}}.
$$
This proves the upper bound in \eqref{eq:L1 embedding}. 

(b) We combine the result of (a) with the Gaussian concentration inequality.
To this end, we must first show that the map $A \mapsto Z = Z(A)$ is Lipschitz where $A = (a_1,\ldots,a_m)$ is considered
as a matrix in the space $\R^{nm}$ equipped with Frobenius norm $\|\cdot\|_F$ (which coincides with the Euclidean norm on $\R^{nm}$).  
It follows from two applications of the triangle inequality followed by two applications of the Cauchy-Schwarz inequality that
for $A = (a_1,\ldots,a_m), \, B = (b_1,\ldots,b_m) \in \R^{nm}$ we have
$$
\abs{Z(A) - Z(B)} \leq \sup_{x \in K} \frac{1}{m} \sum_{i=1}^m \abs{\< a_i - b_i, x \> } \leq \frac{d(K)}{m} \sum_{i=1}^m \twonorm{a_i - b_i} \leq \frac{d(K)}{\sqrt{m}} \fronorm{A - B}.
$$
Thus $Z$ has Lipschitz constant bounded by $d(K)/\sqrt{m}$.  
We may now bound the deviation probability for $Z$ using the Gaussian concentration inequality (see \cite[Equation 1.6]{LT}) as follows:
\[\Pr{\abs{Z - \E Z} \geq u} \leq 2 \exp(-mu^2/2d(K)^2).\]
The deviation inequality \eqref{eq:isoperimetric deviation inequality} now follows from the bound on $\E Z$ from (a).
\end{proof}

\begin{remark}[Random matrix formulation]			\label{rem:rmt}
  One can state Lemma~\ref{lem:concentration} in terms of random matrices. 
  Indeed, let $A$ be an $m \times n$ random matrix with independent $\NN(0,1)$ entries.
  Then its rows $a_i$ satisfy the assumption of Lemma~\ref{lem:concentration}, and 
  we can express $Z$ as 
  \begin{equation}							\label{eq:Z}
  Z = \sup_{x \in K} \Big| \frac{1}{m} \|Ax\|_1 - \sqrt{\frac{2}{\pi}} \|x\|_2 \Big|.
  \end{equation}
\end{remark}

Using this remark for the set $K-K$, we obtain a linear embedding of $K$ into $\ell_1$:

\begin{corollary}[Embedding into $\ell_1$]				\label{cor:into L1}
  Consider a subset $K \subset \ell_2^n$ and let $\delta>0$. Let 
  $$
  m \ge C \delta^{-2} w(K)^2.
  $$
  Then, with probability at least $1 - 2 \exp(-m\delta^2/32)$, the linear map $f : K \to \ell_1^m$ 
  defined as $f(x) = \frac{1}{m} \sqrt{\frac{\pi}{2}} Ax$ is a $\delta$-isometry. 
  Thus $K$ can be linearly embedded into $\ell_1^m$ with Gromov-Haussdorff distortion at most $\delta$. 
\end{corollary}

\begin{proof}
Let $A$ be the random matrix as in Remark~\ref{rem:rmt}. Using Lemma~\ref{lem:concentration}
for $K-K$ and noting the form of $Z$ in \eqref{eq:Z}, we conclude that the following event holds with probability at least $1 - 2 \exp(-m\delta^2/32)$:
$$
\Big| \frac{1}{m} \|Ax-Ay\|_1 - \sqrt{\frac{2}{\pi}} \|x-y\|_2 \Big| 
\le \frac{8w(K-K)}{\sqrt{m}} \le \frac{16 w(K)}{\sqrt{m}} \le \delta,  \qquad x,y \in K.
$$
\end{proof}

\begin{remark}
  The above argument shows in fact that Corollary~\ref{cor:into L1} holds for
  $$
  m \ge C \delta^{-2} w(K-K)^2.
  $$
  As we noticed in Remark~\ref{rem:K-K}, the quantity $w(K-K)$ more accurately reflects 
  the geometric meaning of the mean width than $w(K)$.
\end{remark}

\begin{remark}[Low M$^*$ estimate]
 Note that for the subspace $E = \ker A$ we have from \eqref{eq:Z} that 
 $Z \ge \sup_{x \in K \cap E} \sqrt{\frac{2}{\pi}} \|x\|_2 = \sqrt{\frac{2}{\pi}} \, d(K \cap E)$. 
 Then Lemma~\ref{lem:concentration} implies that 
  \begin{equation}							\label{eq:low M*}
  \E d(K \cap E) \le \frac{6 w(K)}{\sqrt{m}}. 
  \end{equation}  
  By rotation invariance of Gaussian distribution, inequality \eqref{eq:low M*} holds for a random subspace 
  $E$ in $\R^n$ of given codimension $m \le n$, uniformly distributed according to the Haar measure.
  This result recovers (up to the absolute constant $6$ which can be improved) 
  the so-called {\em low M$^*$ estimate} from geometric functional analysis,
  see \cite[Section 15.1]{LT}.
\end{remark}

\begin{remark}[Dimension reduction]
  As we emphasized in the introduction, for many sets $K \subset \R^n$ one has $w(K) \ll n$. 
  In such cases Corollary~\ref{cor:into L1} works for $m \ll n$.
  The embedding of $K$ into $\ell_1^m$ yields dimension reduction for $K$ (from $n$ to $m\ll n$ dimensions). 
  
  For example, if $K$ is a finite set then $w(K) \le C \sqrt{\log |K|}$ (see e.g. \cite[(3.13)]{LT}), 
  and so Corollary~\ref{cor:into L1} applies with $m \sim \log |K|$. 
  This gives the following variant of the {\em Johnson-Lindenstrauss Lemma}: every finite subset of a Euclidean space
  can be linearly embedded in $\ell_1^m$ with $m \sim \log|K|$ and with small distortion in the Gromov-Haussdorff metric.
  Stronger variants of Johnson-Lindenstrauss lemma are known for {\em Lipschitz} rather than Gromov-Haussdorff embeddings
  into $\ell_2^m$ and $\ell_1^m$ \cite{AC, Sch}. However, for general sets $K$ (in particular
  for any set with nonempty interior) a Lipschitz embedding into lower dimensions 
  is clearly impossible; still a Gromov-Haussdorff embedding exists due to Corollary~\ref{cor:into L1}. 
\end{remark}

\section{Proof of Corollary~\ref{cor:cells} by a curvature argument}				\label{sec:curvature}

In this section we give a short  argument that leads to a version of Corollary~\ref{cor:cells} 
with a slightly better dependence of $m$ on $\delta$. 

\begin{theorem}[Cells of random uniform tessellations] 					\label{thm:cells}
  Consider a subset $K \subseteq S^{n-1}$ and let $\delta > 0$. 
  Let 
  $$
  m \ge C \delta^{-4} w(K)^2
  $$  
  and consider an arrangement of $m$ independent random hyperplanes in $\R^n$ that are 
  uniformly distributed according to the Haar measure.
  Then, with probability at least $1-2\exp(-c \delta^4 m)$, all cells of the tessellation have diameter at most $\delta$.
\end{theorem}

The argument is based on Lemma~\ref{lem:concentration}. If points $x,y \in K$ belong 
to the same cell, then the midpoint $z = \frac{1}{2}(x+y)$ also belongs to the same cell (after normalization). 
Using Lemma~\ref{lem:concentration} one can then show that $\|z\|_2 \approx \frac{1}{2}(\|x\|_2 + \|y\|_2) = 1$.
Due to the curvature of the sphere, this forces the length of the interval $\|x-y\|_2$ to be small, 
which means that the diameter of the cell is small. The formal argument is below. 

\begin{proof}
We represent the random hyperplanes as $\{a_i\}^\perp$, where $a_1,\ldots,a_m \sim \NN(0,I_n)$
are independent random vectors in $\R^n$.
Let $\delta, m$ be as in the assumptions of the theorem. 
We shall apply Lemma~\ref{lem:concentration} for the sets $K$ and $\frac{1}{2}(K+K)$
and for $u = \e/2$, where we set $\e = \delta^2/16$. 
Since the diameters of both these sets are bounded by $1$, we obtain that with probability at least $1-2\exp(-c \delta^4 m)$ 
the following event holds:
\begin{equation}							\label{eq:v}
\Big| \sqrt{\frac{\pi}{2}} \frac{1}{m} \sum_{i=1}^m |\< a_i,v\> | - \|v\|_2 \Big| < \e, \qquad
v \in K \cup \frac{1}{2}(K+K).
\end{equation}

Assume that the event \eqref{eq:v} holds. Consider a pair of points $x,y \in K$ that belong to the same cell of the tessellation, 
which means that 
$$
\sign \< a_i,x\> = \sign \< a_i, y\> , \qquad i \in [m].
$$
To complete the proof is suffices to show that $\|x-y\|_2 \le \delta$. 
This will give desired diameter $\delta$ in the Euclidean metric. 
Furthermore, since for small $\delta$ the Euclidean and the geodesic distances are equivalent, the conclusion will hold for the geodesic distance as well. 

We shall use \eqref{eq:v} for 
$x, y \in K$ and for the midpoint $z := \frac{1}{2}(x+y) \in \frac{1}{2}(K+K)$. Clearly 
$\sign \< a_i, z\> = \sign \< a_i,x\> = \sign \< a_i, y\> $, hence 
$$
|\< a_i,z\> | = |\< a_i, x\> | + |\< a_i,y\> |, \qquad i \in [m].
$$ 
Therefore we obtain from \eqref{eq:v} that 
\begin{align}  \label{eq:z large}
\|z\|_2 
  &\ge \sqrt{\frac{\pi}{2}} \frac{1}{m} \sum_{i=1}^m |\< a_i,z\> | - \e		
  = \frac{1}{2} \Big[ \sqrt{\frac{\pi}{2}} \frac{1}{m} \sum_{i=1}^m |\< a_i,x\> | + \sqrt{\frac{\pi}{2}} \frac{1}{m} \sum_{i=1}^m |\< a_i,y\> | \Big] - \e \\
  &\ge \frac{1}{2} (\|x\|_2-\e + \|y\|_2-\e) - \e
  = 1-2\e.  \nonumber
\end{align}
By the parallelogram law, we conclude that 
$$
\|x-y\|_2^2 = 4 - \|x+y\|_2^2
= 4(1-\|z\|_2^2) \le 16\e = \delta^2.
$$
This completes the proof. 
\end{proof}

\subsection{Limitations of the curvature argument}

Unfortunately, the curvature argument does not lend itself to proving the more general result, Theorem~\ref{thm:tessellations}
on uniform tessellations. To see why, suppose $x,y \in K$ do not belong to the same cell but instead 
$d_A(x,y) = d$ for some small $d \in (0,1)$. Consider the set of mismatched signs 
$$
T := \big\{ i \in [m]:\; \sign \< a_i,x\> \ne \sign \< a_i, y\> \big\}; \qquad \frac{|T|}{m} = d.
$$
These signs create an additional error term in the right hand side of \eqref{eq:z large}, which is 
\begin{equation}							\label{eq:mismatch}
\sqrt{\frac{\pi}{2}} \frac{1}{m} \sum_{i \in T} |\< a_i,v_i\> | \qquad \text{where } v_i \in \{x,y\}.
\end{equation}
By analogy with Lemma~\ref{lem:concentration}, we can expect that this term should be approximately 
equal $|T|/m = d$. If this is true, then \eqref{eq:z large} becomes in our situation 
$\|z\|_2 \ge 1 - 2\e - d$, which leads as before to 
$\|x-y\|_2^2 \lesssim \e + d$.
Ignoring $\e$, we see that the best estimate the curvature argument can give is
$d(x,y) \lesssim \sqrt{d_A(x,y)}$
rather than $d(x,y) \lesssim d_A(x,y)$ that is required in Theorem~\ref{thm:tessellations}. 

\smallskip

The weak point of this argument is that it takes into account the size of $T$ but ignores the nature of $T$.
For every $i \in T$, the hyperplane $\{a_i\}^\perp$ passes through the arc connecting $x$ and $y$.
If the length of the arc $d(x,y)$ is small, this creates a strong constraint on $a_i$. 
Conditioning the distribution of $a_i$ on the constraint that $i \in T$ creates a bias toward smaller values of 
$|\< a_i,x\> |$ and $|\< a_i,y\> |$. As a result, the conditional expected value
of the error term \eqref{eq:mismatch} should be smaller than $d$.
Computing this conditional expectation is not a problem for a given pair $x,y$, 
but it seems to be difficult to carry out a uniform argument over $x,y \in K$ 
where the (conditional) distribution of $a_i$ depends on $x,y$. 

We instead propose a different and somewhat more conceptual way to deduce Theorem~\ref{thm:tessellations}
from Lemma~\ref{lem:concentration}. This argument will be developed in the rest of this paper.

\subsection{Dvoretzky theorem and dependence on $\delta$}			\label{sec:Dvoretzky}

The unusual dependence $\delta^{-4}$ in Theorem~\ref{thm:cells} is related to the open problem 
of the optimal dependence on distortion in the Dvoretzky theorem. 

Indeed, consider the special case of the tessellation problem where $K = S^{n-1}$ and $w(K) \sim \sqrt{n}$.  
Then Lemma~\ref{lem:concentration} in its geometric formulation (see equation \eqref{eq:Z} and Corollary~\ref{cor:into L1}) 
states that $\ell_2^n$ embeds into $\ell_1^m$ whenever $m \ge C \e^{-2} n$, meaning that 
$$
(1-\e) \|x\|_2 \le \|\Phi x\|_1 \le (1+\e) \|x\|_2, \qquad x \in \R^n, 
$$
where $\Phi = \sqrt{\frac{\pi}{2}} \frac{1}{m} A$.
Equivalently, there exists an $n$-dimensional subspace of $\ell_1^m$ that is $(1+\e)$-Euclidean, 
where $n \sim \e^2 m$. This result recovers the well known Dvoretzky theorem in V.~Milman's formulation 
(see \cite[Theorem 4.2.1]{GM}) for the space $\ell_1^m$, 
and with the best known dependence on $\e$. However, it is not known whether $\e^2$ is the optimal dependence for $\ell_1^m$; 
see \cite{Sch} for a discussion of the general problem of dependence on $\e$ in Dvoretzky theorem.
 
These observation suggest that we can reverse our logic. Suppose one can prove Dvoretzky theorem for $\ell_1^m$ with a better dependence
on $\e$, thereby constructing a $(1+\e)$-Euclidean subspace of dimension $n \sim f(\e) m$ with $f(\e) \gg \e^2$. 
Then such construction can replace Lemma~\ref{lem:concentration} in the curvature argument.
This will lead to Theorem~\ref{thm:cells} for $K = S^{n-1}$ with an improved dependence on $\delta$, namely with 
$m \sim f(\delta^2) n$.
Concerning lower bounds, the best possible dependence of $m$ on $\delta$ should be $\delta^{-1}$, which follows by considering the case $n=2$. 
This dependence will be achieved if Dvoretzky theorem for $\ell_1^m$ is valid with $n \sim \e^{1/2} m$. This is unknown.

\section{Toward Theorem~\ref{thm:tessellations}: a soft Hamming distance}			\label{sec:soft}

Our proof of Theorem~\ref{thm:tessellations} will be based on a covering argument. 
A standard covering argument of geometric functional analysis would proceed in our situation as follows:

\begin{enumerate}[(a)]
  \item Show that $d_A(x,y) \approx d(x,y)$ with high probability for a fixed pair $x,y$. 
    This can be done using standard concentration inequalities.
  \item Prove that $d_A(x,y) \approx d(x,y)$ uniformly for all $x,y$ in a finite $\e$-net $N_\e$ of $K$. 
    Sudakov's inequality can be used to estimate the cardinality of $N_\e$ via the mean width $w(K)$. 
    The conclusion will follow from step 1 by the union bound over $(x,y) \in N_\e \times N_\e$. 
  \item Extend the estimate $d_A(x,y) \approx d(x,y)$ from $x,y \in N_\e$ to $x,y \in K$ by approximation.
\end{enumerate}

While the first two steps are relatively standard, step (c) poses a challenge in our situation. The Hamming distance $d_A(x,y)$ 
is a discontinuous function of $x,y$, so it is not clear whether the estimate $d_A(x,y) \approx d(x,y)$
can be extended from a pair points $x,y \in N_\e$ to a pair of nearby points. In fact, for some tessellations this 
task is impossible. Figure~\ref{fig:non-uniform-tessellation} shows that there exist very non-uniform tessellations
that are nevertheless very uniform for an $\e$-net, namely one has $d_A(x,y) = d(x,y)$ for all $x,y \in N_\e$.
The set $K$ in that example is a subset of the plane $\R^2$, and one can clearly embed such a set with 
into the sphere $S^2$ as well.

\begin{figure}[htp]		
  \centering \includegraphics[height=2.2cm]{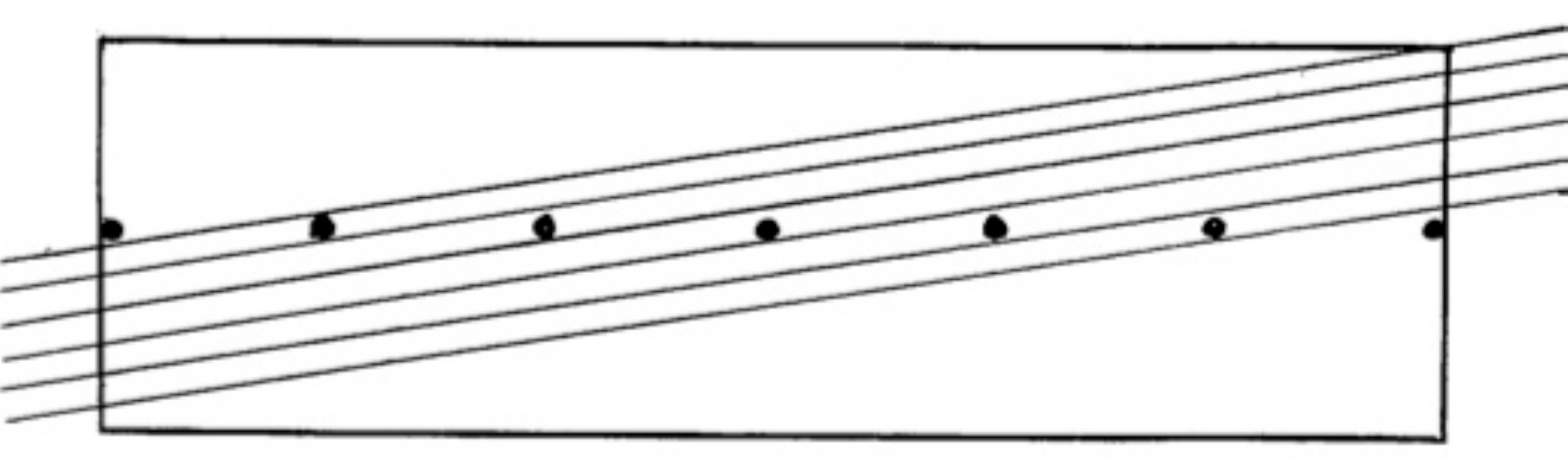} 
  \caption{This hyperplane tessellation of the set $K =[-\frac{1}{2}, \frac{1}{2}] \times [-\frac{\e}{2},\frac{\e}{2}]$ is very non-uniform,
  as all cells have diameter at least $1$. 
  The tessellation is nevertheless very uniform for the $\e$-net $N_\e = \e\Z \cap K$, as $d_A(x,y) = \|x-y\|_2$ for all $x,y \in N_\e$.}
  \label{fig:non-uniform-tessellation}
\end{figure}

To overcome the discontinuity problem, we propose to work with a soft version of the Hamming distance.
Recall that $m$ hyperplanes are determined by their normals $a_1,\ldots,a_m \in \R^n$, which we 
organize in an $m \times n$ matrix $A$ with rows $a_i$.
Then the usual (``hard'') Hamming distance $d_A(x,y)$ on $\R^n$ with respect to $A$ with can be expressed as
\begin{equation}							\label{eq:hard}
d_A(x,y) = \frac{1}{m} \sum_{i=1}^m \one_{\EE_i}, \quad  
\text{where} \quad \EE_i = \{ \sign\< a_i, x\> \ne \sign\< a_i, y\> \}. 
\end{equation}

\begin{definition}[Soft Hamming distance]
  Consider an $m \times n$ matrix $A$ with rows $a_1,\ldots,a_m$, and let $t \in \R$. 
  The {\em soft Hamming distance} $d_A^t(x,y)$ on $\R^n$ is defined as 
  \begin{gather}
  d_A^t(x,y) = \frac{1}{m} \sum_{i=1}^m \one_{\FF_i}, \quad  \text{where} \nonumber\\
  \FF_i = \{ \< a_i, x\> > t, \; \< a_i, y\> < -t \} \cup \{ -\< a_i, x\> > t, \; -\< a_i, y\> < -t \}.		\label{eq:soft}
  \end{gather}
\end{definition}

Both positive and negative $t$ may be considered. For positive $t$ the soft Hamming distance
counts the hyperplanes that separate $x,y$ well enough; for negative $t$ it counts the hyperplanes
that separate or nearly separate $x,y$.

\begin{remark}[Comparison of soft and hard Hamming distances]
  Clearly $d_A^t(x,y)$ is a non-increasing function of $t$. Moreover, 
  \begin{align*}
  d_A^t(x,y) = d_A(x,y) \quad &\text{for } t = 0; \\
  d_A^t(x,y) \le d_A(x,y) \quad &\text{for } t \ge 0; \\
  d_A^t(x,y) \ge d_A(x,y) \quad &\text{for } t \le 0.
  \end{align*}
\end{remark}

The soft Hamming distance for a fixed $t$ is as discontinuous as the usual (hard) Hamming distance. 
However, some version of continuity emerges when we allow $t$ to vary slightly:

\begin{lemma}[Continuity]					\label{lem:continuity}
  Let $x,y,x',y' \in \R^n$, and assume that $\|Ax'\|_\infty \le \e$, $\|Ay'\|_\infty \le \e$ for some $\e>0$.
  Then for every $t \in \R$ one has
  $$
  d_A^{t+\e}(x,y) \le d_A^t(x+x',y+y') \le d_A^{t-\e}(x,y).
  $$  
\end{lemma}

\begin{proof}
Consider the events $\FF_i = \FF_i (x,y,t)$ from the definition of the soft Hamming distance \eqref{eq:soft}.
By the assumptions, we have $|\< a_i, x'\> | \le \e$, $|\< a_i, y'\> | \le \e$ for all $i \in [m]$.
This implies by the triangle inequality that 
$$
\FF_i(x,y,t+\e) \subseteq \FF_i(x+x',y+y',t) \subseteq \FF_i(x,y,t-\e).
$$
The conclusion of the lemma follows.
\end{proof}

We are ready to state a stronger version of Theorem~\ref{thm:tessellations} for the soft Hamming distance. 

\begin{theorem}[Random uniform tessellations: soft version]		\label{thm:tessellations soft}
  Consider a subset $K \subseteq S^{n-1}$ and let $\delta > 0$. 
  Let 
  $$
  m \ge C \delta^{-6} w(K)^2
  $$ 
  and pick $t \in \R$.
  Consider an $m \times n$ random (Gaussian) matrix $A$ with independent rows $a_1,\ldots,a_m \sim \NN(0,I_n)$.  
  Then with probability at least $1-\exp(-c \delta^2 m)$, one has
  $$
  | d_A^t(x,y) - d(x,y) | \le \delta + 2|t|, \quad x,y \in K.
  $$
\end{theorem}

Note that if we take $t = 0$ in the above theorem, we recover Theorem~\ref{thm:tessellations}.  
However, we find it easier to prove the result for general $t$, since in our argument we will work with 
different values of the $t$ for the soft Hamming distance.

Theorem~\ref{thm:tessellations soft} is proven in the next section.

\section{Proof of Theorem~\ref{thm:tessellations soft} on the soft Hamming distance}

We will follow the covering argument outlined in the beginning of Section~\ref{sec:soft},
but instead of $d_A(x,y)$ we shall work with the soft Hamming distance $d_A^t(x,y)$.

\subsection{Concentration of distance for a given pair}

At the first step, we will check that $d_A^t(x,y) \approx d(x,y)$ with high probability for a fixed pair $x,y$.
Let us first verify that this estimate holds in expectation, i.e. that $\E d_A^t(x,y) \approx d(x,y)$. 
One can easily check that
\begin{equation}							\label{eq:Ed again}
\E d_A(x,y) = d(x,y),
\end{equation}
so we may just compare $\E d_A^t(x,y)$ to $\E d_A(x,y)$.
Here is a slightly stronger result:

\begin{lemma}[Comparing soft and hard Hamming distances in expectation]		\label{eq:Edt Ed}
  Let $A$ be a random Gaussian matrix be as in Theorem~\ref{thm:tessellations soft}. 
  Then, for every $t \in \R$ and every $x,y \in \R^n$, one has 
  $$
  | \E d_A^t(x,y) - d(x,y) | \le \E |d_A^t (x,y) - d_A(x,y)| \le 2|t|.
  $$
\end{lemma}

\begin{proof}
The first inequality follows from \eqref{eq:Ed again} and Jensen's inequality.
To prove the second inequality, we use the events $\EE_i$ and $\FF_i$ from Equations \eqref{eq:hard}, \eqref{eq:soft}
defining the hard and soft Hamming distances, respectively. It follows that
\begin{align*}
\E |d_A^t (x,y) - d_A(x,y)| 
  &= \E \Big| \frac{1}{m} \sum_{i=1}^m (\one_{\EE_i} - \one_{\FF_i}) \Big| \\
  &\le \E |\one_{\EE_1} - \one_{\FF_1}|	\qquad \text{(by triangle inequality and identical distribution)} \\
  &= \P \{ \EE_1 \bigtriangleup \FF_1\}  \\
  &\le \P \{ |\< a_1,x\> | \le |t| \} + \P \{ |\< a_1,y\> | \le |t| \} \\
  &\le 2 \P \{ |g| \le |t| \}		\qquad \text{(where $g \sim \NN(0,1)$)} \\
  &\le 2|t| 		\qquad \text{(by the density of the normal distribution).}	\qedhere
\end{align*}
\end{proof}

Now we upgrade Lemma~\ref{eq:Edt Ed} to an concentration inequality:

\begin{lemma}[Concentration of distance]			\label{lem:concentration dist}
  Let $A$ be a random Gaussian matrix as in Theorem~\ref{thm:tessellations soft}. 
  Then, for every $t \in \R$ and every $x,y \in \R^n$, the following deviation inequality holds: 
  $$
  \P \big\{ |d_A^t(x,y) - d(x,y)| > 2|t| + \delta \big\} \le 2 \exp(-2 \delta^2 m), \quad \delta > 0.
  $$
\end{lemma}

\begin{proof}
By definition, $m \cdot d_A^t(x,y)$ has the binomial distribution $\text{Bin}(m,p)$.
The parameter $p = \E d_A^t(x,y)$ satisfies by Lemma~\ref{eq:Edt Ed} that 
$$
|p-d(x,y)| \le 2|t|.
$$
A standard Chernoff bound for binomial random variables states that
$$
\P \big\{ |d_A^t(x,y) - p| > \delta \big\} \le 2 \exp(-2 \delta^2 m), \quad \delta > 0,
$$
see e.g. \cite[Corollary~A.1.7]{AS}.
The triangle inequality completes the proof.
\end{proof}

\subsection{Concentration of distance over an $\e$-net}			\label{sec:over net}

Let us fix a small $\e>0$ whose value will be determined later. 
Let $N_\e$ be an $\e$-net of $K$ in the Euclidean metric. By Sudakov's inequality
(see \cite[Theorem~3.18]{LT}), we can arrange the cardinality of $N_\e$ to satisfy
\begin{equation}							\label{eq:net}
\log |N_\e| \le C \e^{-2} w(K)^2.
\end{equation}
We can decompose every vector $x \in K$ into a {\em center} $x_0$ and
a {\em tail} $x'$ so that
\begin{equation}							\label{eq:decomposition}
x = x_0 + x', \quad \text{where} \quad x_0 \in N_\e, \quad x' \in (K-K) \cap \e B_2^n.
\end{equation}
We first control the centers by taking a union bound in Lemma~\ref{lem:concentration dist} 
over the net $N_\e$:

\begin{lemma}[Concentration of distance over a net]			\label{lem:concentration net}
  Let $A$ a random Gaussian matrix be as in Theorem~\ref{thm:tessellations soft}. 
  Let $N_\e$ be a subset of $S^{n-1}$ whose cardinality satisfies \eqref{eq:net}.
  Let $\delta>0$, and assume that
  \begin{equation}							\label{eq:concentration net m}
  m \ge C \e^{-2} \delta^{-2} w(K)^2.
  \end{equation}
  Let $t \in \R$. Then the following holds with probability at least $1 - 2\exp(-\delta^2 m)$:
  $$
  |d_A^t(x_0,y_0) - d(x_0,y_0)| \le 2|t| + \delta, \quad x_0,y_0 \in N_\e.
  $$
\end{lemma}

\begin{proof}
By Lemma~\ref{lem:concentration net} and a union bound over the set of pairs $(x_0,y_0) \in N_\e \times N_\e$, we obtain
$$
\P \Big\{ \sup_{x,y \in N_\e} |d_A^t(x,y) - d(x,y)| > 2|t| + \delta \Big\}
\le |N_\e|^2 \cdot 2 \exp(-2 \delta^2 m)
\le 2 \exp(-\delta^2 m)
$$
where the last inequality follows by \eqref{eq:net} and \eqref{eq:concentration net m}.
The proof is complete.
\end{proof}

\subsection{Control of the tails}

Now we control the tails $x' \in (K-K) \cap \e B_2^n$ in decomposition \eqref{eq:decomposition}.

\begin{lemma}[Control of the tails]					\label{lem:tails}
  Consider a subset $K \subseteq S^{n-1}$ and let $\e>0$. Let 
  $$
  m \ge C \e^{-2} w(K)^2.
  $$ 
  Consider independent random vectors $a_1,\ldots,a_m \sim \NN(0,I_n)$.  
  Then with probability at least $1-2\exp(-cm)$, one has
  $$
  \frac{1}{m} \sum_{i=1}^m |\< a_i,x'\> | \le \e \quad \text{for all } x' \in (K-K) \cap \e B_2^n.
  $$
\end{lemma}

\begin{proof}
Let us apply Lemma~\ref{lem:concentration} for the set $T = (K-K) \cap \e B_2^n$ instead of $K$, 
and for $u = \e/8$. Since $d(K) = \max_{x' \in T} \|x'\|_2 \le \e$, we obtain that the following holds with probability at least 
$1 - 2\exp(-cm)$:
\begin{align}
\sup_{x' \in T} \frac{1}{m} \sum_{i=1}^m |\< a_i,x'\> | 
  &\le \sup_{x' \in T} \Big| \frac{1}{m} \sum_{i=1}^m |\< a_i,x'\> | - \sqrt{\frac{2}{\pi}} \|x'\|_2 \Big| + \sqrt{\frac{2}{\pi}} \, \e \nonumber\\
  &\le \frac{4 w(T)}{\sqrt{m}} + \frac{\e}{8} + \sqrt{\frac{2}{\pi}} \, \e.				\label{eq:over T}
\end{align}
Note that $w(T) \le w(K-K) \le 2w(K)$. So using the assumption on $m$ we conclude that the quantity in \eqref{eq:over T} 
is bounded by $\e$, as claimed. 
\end{proof}

\subsection{Approximation}

Now we establish a way to transfer the distance estimates from an $\e$-net $N_\e$ to the full set $K$.
This is possible by a continuity property of the soft Hamming distance, which we outlined in Lemma~\ref{lem:continuity}.
This result requires the perturbation to be bounded in $L_\infty$ norm. However, in our situation 
the perturbations are going to be bounded only in $L_1$ norm due to Lemma~\ref{lem:tails}.
So we shall prove the following relaxed version of continuity:

\begin{lemma}[Continuity with respect to $L_1$ perturbations]					\label{lem:continuity L1}
  Let $x,y,x',y' \in \R^n$, and assume that $\|Ax'\|_1 \le \e m$, $\|Ay'\|_1 \le \e m$ for some $\e>0$.
  Then for every $t \in \R$ and $M \ge 1$ one has
  \begin{equation}							\label{eq:continuity L1}
  d_A^{t+M\e}(x,y) - \frac{2}{M} \le d_A^t(x+x',y+y') \le d_A^{t-M\e}(x,y) + \frac{2}{M}.
  \end{equation}
\end{lemma}

\begin{proof}
Consider the events $\FF_i = \FF_i (x,y,t)$ from the definition of the soft Hamming distance \eqref{eq:soft}.
By the assumptions, we have 
$$
\sum_{i=1}^m |\< a_i, x'\> | \le \e m, \quad
\sum_{i=1}^m |\< a_i, y'\> | \le \e m.
$$
Therefore, the set 
$$
T := \big\{ i \in [m]:\; |\< a_i, x'\> | \le M\e, \; |\< a_i, y'\> | \le M\e \big\}
\quad \text{satisfies} \quad |T^c| \le 2m/M.
$$
By the triangle inequality, we have
$$
\FF_i(x,y,t+M\e) \subseteq \FF_i(x+x',y+y',t) \subseteq \FF_i(x,y,t-M\e), \quad i \in T.
$$
Therefore
\begin{align*}
d_A^{t+M\e}(x,y) 
  &= \frac{1}{m} \sum_{i=1}^m \one_{\FF_i(x,y,t+M\e)}
  \le \frac{|T^c|}{m} + \frac{1}{m} \sum_{i \in T} \one_{\FF_i(x,y,t+M\e)} \\
  &\le \frac{2}{M} + \frac{1}{m} \sum_{i \in T} \one_{\FF_i(x+x',y+y',t)} 
  \le \frac{2}{M} + d_A^t(x+x',y+y').
\end{align*}
This proves the first inequality in \eqref{eq:continuity L1}. The proof of the second inequality is similar.
\end{proof}

\subsection{Proof of Theorem~\ref{thm:tessellations soft}.}				\label{sec:proof tessellations}

Now we are ready to combine all the pieces and prove Theorem~\ref{thm:tessellations soft}.
To this end, consider the set $K$, numbers $\delta$, $m$, $t$, and the random matrix $A$ as in the theorem.
Choose $\e = \delta^2/100$ and $M = 10/\delta$.

Consider an $\e$-net $N_\e$ of $K$ as we described in the beginning of Section~\ref{sec:over net}.
Let us apply Lemma~\ref{lem:concentration net} that controls the distances on $N_\e$
along with Lemma~\ref{lem:tails} that controls the tails. 
By the assumption on $m$ in the theorem and by our choice of $\e$, both requirements on $m$ in these lemmas hold.
By a union bound, with probability at least $1-4\exp(-c\delta^2 m)$ the following event holds:
for every $x_0, y_0 \in N_\e$ and $x',y' \in (K-K) \cap \e B_2^n$, one has 
\begin{gather}
|d_A^{t-M\e}(x_0,y_0) - d(x_0,y_0)| \le 2|t-M\e| + \delta/2, 		\label{eq:t-Me}\\
|d_A^{t+M\e}(x_0,y_0) - d(x_0,y_0)| \le 2|t+M\e| + \delta/2, \nonumber\\
\|Ax'\|_1 \le \e m, \quad \|Ay'\|_1 \le \e m.		\label{eq:tail bounds}
\end{gather}

Let $x,y \in K$. As we described in \eqref{eq:decomposition}, we can decompose the vectors as 
\begin{equation}							\label{eq: decomposition x y}
x = x_0 + x', \quad y = y_0 + y', \quad \text{where} \quad x_0,y_0 \in N_\e, \quad x',y' \in (K-K) \cap \e B_2^n.
\end{equation}
The bounds in \eqref{eq:tail bounds} guarantee that the continuity property \eqref{eq:continuity L1} 
in Lemma~\ref{lem:continuity L1} holds. This gives
\begin{align*}
d_A^t(x,y) 
  &\le d_A^{t-M\e}(x_0,y_0) + \frac{2}{M} \\
  &\le d(x_0,y_0) + 2|t| + 2M\e + \frac{\delta}{2} + \frac{2}{M}		\qquad \text{(by \eqref{eq:t-Me} and the triangle inequality).}
\end{align*}
Furthermore, using \eqref{eq: decomposition x y} we have
$$
|d(x_0,y_0) - d(x,y)| \le d(x_0,x) + d(y_0,y) \le \|x_0-x\|_2 + \|y_0-y\|_2 
\le 2 \e. 
$$
It follows that 
$$
d_A^t(x,y) 
\le d(x,y) + 2|t| + 2M\e + \frac{\delta}{2} + \frac{2}{M} + 2\e.
$$
Finally, by the choice of $\e$ and $M$ we obtain
$$
d_A^t(x,y) \le d(x,y) + 2|t| + \delta.
$$

A similar argument shows that 
$$
d_A^t(x,y) \ge d(x,y) - 2|t| - \delta.
$$
We conclude that 
$$
| d_A^t(x,y) - d(x,y) | \le \delta + 2|t|.
$$
This completes the proof of Theorem~\ref{thm:tessellations soft}. \qed

\section{Proof of Theorem~\ref{thm:tessellations Rn} on tessellations in $\R^n$}				\label{sec:Rn}

In this section we deduce Theorem~\ref{thm:tessellations Rn} from Theorem~\ref{thm:tessellations}
by an elementary lifting argument into $\R^{n+1}$. 
We shall use the following notation: Given a vector $x \in \R^n$ and a number $t \in \R$, 
the vector $x \oplus t \in \R^{n} \oplus \R = \R^{n+1}$ is the concatenation of $x \in \R^n$
and $t$. Furthermore, $K \oplus t$ denotes the set of all vectors $x \oplus t$ where $x \in K$.
 
Assume $K \subset \R^n$ has $\diam(K) = 1$. Translating $K$ if necessary we may assume 
that $0 \in K$; then 
\begin{equation}										\label{eq:radius}
\frac{1}{2} \le \sup_{x \in K} \|x\|_2 \le 1. 
\end{equation}
Also note that by assumption we have 
\begin{equation}							\label{eq:m affine}
m \ge C \delta^{-12} w(K-K) \ge C \delta^{-12} w(K).  
\end{equation}

Fix a large number $t \ge 2$ whose value will be chosen later and consider the set 
$$
K' = Q(K \oplus t) \subseteq S^{n}
$$
where $Q: \R^{n+1} \to S^n$ denotes the spherical projection map $Q(u) = u/\|u\|_2$. 
We have 
\begin{align*}
w(K') 
  &\le t^{-1} w(K \oplus t)		\qquad \text{(as $\|u\|_2 \ge t$ for all $u \in K \oplus t$)} \\
  &\le t^{-1} (w(K) + t \E|\gamma|)  		\qquad \text{(where $\gamma \sim \NN(0,1)$)} \\
  &= t^{-1} w(K) + \sqrt{2/\pi} \le 3w(K) 
\end{align*}
where the last inequality holds because $w(K) \ge \sqrt{2/\pi} \sup_{x \in K} \|x\|_2 \ge 1/\sqrt{2\pi}$
by \eqref{eq:radius}.

Then Theorem~\ref{thm:tessellations} implies that if $m \ge C \delta_0^{-6} w(K)^2$ for some $\delta_0>0$,
then there exists an arrangement of $m$ hyperplanes in $\R^{n+1}$ such that 
\begin{equation}										\label{eq:x'y'}
|d_A(x',y') - d(x',y')| \le \delta_0, \quad x',y' \in K'.
\end{equation}
Consider arbitrary vectors $x$ and $y$ in $K$ and the corresponding vectors $x' = Q(x \oplus t)$ and
$y' = Q(x \oplus t)$ in $K'$. Let us relate the distances between $x'$ and $y'$ 
appearing in \eqref{eq:x'y'} to corresponding distances between $x$ and $y$.

Let $a_i \oplus a \in \R^{n+1}$ denote normals of the hyperplanes. 
Clearly, $x'$ and $y'$ are separated by the $i$-th hyperplane if and only if $x \oplus t$ and $y \oplus t$ are.
This in turn happens if and only if $x$ and $y$ are separated by the affine hyperplane
that consists of all $x \in \R^n$ satisfying $\< a_i \oplus a, x \oplus t\> = \< a_i, x\> + at = 0$.
In other words, the hyperplane tessellation of $K'$ induces an {\em affine} hyperplane tessellation of $K$,  
and the fraction $d_A(x',y')$ of the hyperplanes separating $x'$ and $y'$ equals the fraction of 
the affine hyperplanes separating $x$ and $y$. With a slight abuse of notation, 
we express this observation as
\begin{equation}							\label{eq:linear affine}
d_A(x',y') = d_A(x,y).
\end{equation}

Next we analyze the normalized geodesic distance $d(x',y')$, which satisfies
\begin{equation}										\label{eq:geodesic Euclidean}
\big| \pi \cdot d(x',y') - \|x'-y'\|_2 \big| \le C_0 \|x'-y'\|_2^2.
\end{equation}
Denoting $t_x = \|x \oplus t\|_2$ and $t_y = \|y \oplus t\|_2$ and using the triangle inequality, we obtain
\begin{align}
\e := \big| \|x'-y'\|_2 - t^{-1} \|x-y\|_2 \big|
  &= \big| \big\| t_x^{-1} (x \oplus t) - t_y^{-1} (y \oplus t) \big\|_2 - \|t^{-1} x - t^{-1} y\|_2 \big| \nonumber \\
  &\le \|x\| \, |t_x^{-1} - t^{-1}| + \|y\| \, |t_y^{-1} - t^{-1}| + t \, |t_x^{-1} - t_y^{-1}|.		\label{eq:contracted dist}
\end{align}
Note that \eqref{eq:radius} yields that $t \le t_x, t_y \le \sqrt{t^2+1}$. It follows that 
$|t_x^{-1} - t^{-1}| \le 0.5 t^{-3}$ and the same bound holds for the other two similar terms in \eqref{eq:contracted dist}. 
Using this and \eqref{eq:radius} we conclude that $\e \le t^{-2}$.
Putting this into \eqref{eq:geodesic Euclidean} and using the triangle inequality twice, we obtain
$$
\big| \pi \cdot d(x',y') - t^{-1} \|x-y\|_2 \big| 
\le C_0 \big( t^{-1} \|x'-y'\|_2 + \e \big)^2 + \e
\le C_0 \big( 2t^{-1} + t^{-2} \big)^2 + t^{-2} 
\le C_1 t^{-2}.
$$

Finally, we use this bound and \eqref{eq:linear affine} in \eqref{eq:x'y'}, which gets us
\begin{equation}							\label{eq:final dist}
\big| \pi t \cdot d_A(x,y) - \|x-y\|_2 \big| \le \pi t \delta_0 + C_1 t^{-1}.
\end{equation}
Now we can assign the values $t := 2C_1/\delta$ and $\delta_0 = \delta^2/(4\pi C_1)$ so the right hand side of \eqref{eq:final dist}
is bounded by $\delta$, as required. 
Note that the condition $m \ge C \delta_0^{-6} w(K)^2$ that we used above in order to apply Theorem~\ref{thm:tessellations}
is satisfied by \eqref{eq:m affine}. This completes the proof of Theorem~\ref{thm:tessellations Rn}. \qed

\end{document}